\documentclass[11pt,a4paper,reqno]{amsart}
\usepackage{amsmath,amscd,amssymb,latexsym}
\usepackage{hyperref}
\usepackage{indentfirst}
\usepackage{mathtools}
\usepackage{enumitem}
\usepackage{float}
\usepackage{multirow}
\usepackage{makecell}
\usepackage{xypic}

\usepackage[last]{changes}

\newtheorem{Thm}{Theorem}[section]

\newtheorem{Lem}[Thm]{Lemma}

\newtheorem{Prob}[Thm]{Problem}
\theoremstyle{definition}

\newtheorem{Ques}[Thm]{Question}
\theoremstyle{remark}
\newtheorem{Rk}[Thm]{Remark}

\def\C{{\mathbb C}}
\def\F{{\mathbb F}}

\def\P{{\mathbb P}}
\def\Q{{\mathbb Q}}

\begin{document}
\title[Jordan constants of abelian surfaces over finite fields]{Jordan constants of abelian surfaces over finite fields}

\author[W. Hwang and B.-H. Im]{WonTae Hwang and Bo-Hae Im}
\date{\today}
\thanks{WonTae Hwang was supported by research funds for newly appointed professors of Jeonbuk National University in 2021. Bo-Hae Im was supported by the National Research Foundation of Korea(NRF) grant funded by the Korea government(MSIT) (No.~2020R1A2B5B01001835).}
\address{WonTae Hwang, Department of Mathematics, and Institute of Pure and Applied Mathematics, Jeonbuk National University, Baekje-daero, Deokjin-gu, Jeonju-si, Jeollabuk-do, 54896, South Korea}
\email{hwangwon@jbnu.ac.kr}
\address{Bo-Hae Im, Dept. of Mathematical Sciences, KAIST, 291 Daehak-ro, Yuseong-gu, Daejeon, South Korea, 34141}
\email{bhim@kaist.ac.kr}

\begin{abstract}
We compute the exact values of the Jordan constants of abelian surfaces over finite fields.
\end{abstract}

\maketitle

\begin{center}
{\it Dedicated to Professor Michael J. Larsen \\on the occasion of his 60th birthday}
\end{center}

\section{Introduction}\label{intro}


Given a group $G$, there are various properties of $G$, in which one is usually interested. Among them, the notion of $G$ being a \emph{Jordan group} was introduced by Vladimir Popov in \cite{Popov2011}. We recall that $G$ is called a \emph{Jordan group} (or said to \emph{satisfy the Jordan property}) if there exists an integer $d>0$ such that every finite subgroup $H$ of $G$ contains a normal abelian subgroup whose index in $H$ is less than or equal to $d.$ The minimal such an integer $d$ is called the \emph{Jordan constant of $G$} and is denoted by $J_{G}$. As a motivating example, a classical result of Camille Jordan says that for an integer $n \geq 1$ the general linear group $\textrm{GL}_n(k)$ over an algebraically closed field $k$ of characteristic zero is a Jordan group, which, in turn, implies that every affine algebraic group is also a Jordan group. Other interesting examples arise in connection with (birational) algebraic geometry. Let $X$ be an irreducible algebraic variety over an algebraically closed field $k$ of characteristic zero, and let $G=\textrm{Bir}(X)$, the group of birational automorphisms of $X$. We might ask whether $G$ is a Jordan group or not. If $X$ is a smooth projective rational variety of dimension $n$ (i.e.\ $X$ is birational to the projective space $\P_k^n$), then $G=\textrm{Cr}_n(k),$ the \emph{Cremona group of degree $n$ over $k,$} and it was known that $\textrm{Cr}_n (k)$ is a Jordan group for $n=1$, $2$ (see \cite[Theorem 5.3]{Ser2009}), and $n=3$ (see \cite[Corollary 1.9]{PS2016}). Yuri Prokhorov and Constantin Shramov \cite{PS2016,PS2017} showed (modulo the Borisov-Alexeev-Borisov conjecture, see \cite[Conjecture 1.7]{PS2016}, recently proved by Caucher Birkar \cite{CB2016}) that $\textrm{Cr}_n$ is indeed a Jordan group for every $n > 3$, and computed an upper bound for the Jordan constants of the Cremona groups of rank $2$ and $3.$ The exact value of the Jordan constant of $\textrm{Cr}_2 (k)$ was obtained by Egor Yasinsky \cite{Ya2017}. As for the case of the automorphism groups, Sheng Meng and De-Qi Zhang \cite{MD2018} proved that the full automorphism group of any projective variety over $k$ is a Jordan group, using algebraic group theoretic arguments. There are also some results on the criterion for determining the Jordaness of $G=\textrm{Bir}(X)$ in small dimensions. For example, if $\textrm{dim}~X \leq 2,$ then it is known \cite[Theorem 2.32]{Popov2011} that the group $\textrm{Bir}(X)$ is a Jordan group if and only if $X$ is not birationally isomorphic to $\P^1_k \times E$ where $E$ is an elliptic curve over $k.$ In addition, we have a similar but slightly more complicated result for the case of $X$ being a threefold \cite[Theorem 1.8]{PS2018}. \\
\indent On the other hand, it seems to the authors that less is known in the case when the base field $k$ is not algebraically closed or the characteristic of the base field $k$ is positive. As an example along this line, Fei Hu \cite{Hu2020} showed two Jordan-type properties (one of which is essentially motivated by a beautiful work of Michael Larsen and Richard Pink \cite{LP2011}) of the automorphism groups of any projective varieties in positive characteristic case. Also, the first author \cite{Hwa2022} (resp.\ \cite{Hwa2023}) computed the rather exact values of the Jordan constants of the multiplicative subgroups of the endomorphism algebras of simple abelian surfaces over fields of positive characteristic (resp.\ simple abelian fourfolds over finite fields). \\

The primary goal of this paper is to extend the results in \cite{Hwa2022} and consider the case of $X$ being an arbitrary abelian surface over a finite field. To this aim, we first recall that there is a classification of finite groups that can be realized as the full automorphism groups of arbitrary polarized abelian surfaces over finite fields (see Section \ref{aut gp} below).
Using this, we can answer one interesting special case of the following problem.
\begin{Prob}\label{problem}
Let $X$ be a (unpolarized) abelian variety over a field $k$ of dimension $g$ with $\textrm{char}(k)=p>0$. Let $\textrm{End}_k^0(X)$ be the endomorphism algebra of $X$ over $k$. \\
(1) Is the multiplicative subgroup of $\textrm{End}_k^0(X)$ a Jordan group? \\
(2) If the answer for (1) is yes, then what is the exact value of its Jordan constant?
\end{Prob}
More precisely, in this paper, we try to answer Problem \ref{problem} for the case when $g=2$ and $k$ is a finite field. To achieve the goal, in view of \cite[Corollary 2 in $\S19$]{Mum2008}, we first consider the cases of $G$ being the multiplicative subgroup of a direct product of two imaginary quadratic fields or an imaginary quadratic field and a totally definite quaternion algebra over $\Q$, or else a power of a totally definite quaternion algebra over $\Q$ to prove that $G$ is a Jordan group, and we also explicitly compute the Jordan constant $J_G$ of $G$ in each case. This gives another example that connects the theory of Jordan groups, algebraic geometry, and algebraic number theory. In this aspect, one of our main results is the following theorem.
\begin{Thm}\label{main theorem1}
Let $E_1$ and $E_2$ be two non-isogenous elliptic curves over a finite field $k=\mathbb{F}_q$ where $q=p^a$ for a prime $p$ with $D_1 = \textrm{End}_k^0(E_1)$ and $D_2 = \textrm{End}_k^0 (E_2)$, and let $X=E_1 \times E_2$ with $D:=\textrm{End}_k^0(X)$. Then $G:=D^{\times}$ is a Jordan group and we have
\begin{equation*}
	J_{G} \in \begin{cases} \{1, 12, 144 \} & \mbox{if $p=2$}; \\ \{1, 2, 4\}  & \mbox{if $p=3$}; \\ \{1\}  & \mbox{if $p \geq 5$} .\end{cases}
	\end{equation*}
\end{Thm}
The proof of Theorem \ref{main theorem1} combines the theory of finite subgroups of quaternion algebras over $\Q$ with some property of the Jordan constant under a direct product of two groups. For more details, see Theorem \ref{prodnonisoellip} below.

On the other hand, if an abelian surface $X$ is a power of an elliptic curve $E$, then it is well known that the endomorphism algebra of $X$ is a matrix algebra over the endomorphism algebra of $E$, and hence, by considering all possible endomorphism algebras of $E,$ we obtain another main result of this paper as follows.

\begin{Thm}\label{main theorem2}
Let $E$ be an elliptic curve over a finite field $k=\mathbb{F}_q$ where $q=p^a$ for a prime $p$, and let $X=E^2$ with $D:=\textrm{End}_k^0(X).$ Then $G:=D^{\times}$ is a Jordan group and we have
\begin{equation*}
	J_{G} \in \begin{cases} \{ 2, 12, 24\} & \mbox{if $E$ is ordinary}; \\ \{12, 24, 60, 120, 360, 960 \}  & \mbox{if $E$ is supersingular},\end{cases}
	\end{equation*}
(where in the second case, all of the endomorphisms of $E$ are defined over $k$).
\end{Thm}

The proof of Theorem \ref{main theorem2} requires concrete knowledge of finite subgroups of $GL_2(D)$, where $D$ is either an imaginary quadratic field or a totally definite quaternion algebra over $\Q.$ In particular, for the latter case, we need to make use of quaternionic matrix representation theory.

For more details, see Theorems \ref{powordelli} and \ref{pow of supell} below.

\begin{Rk}
In fact, in light of the proofs of \cite[Theorems 6.6, 6.8, and 6.9]{Hwa2021}, we can see that the converses of Theorems \ref{main theorem1} and \ref{main theorem2} also hold in the sense that all the integers in the lists are realizable as the Jordan constants of the multiplicative subgroups of the endomorphism algebras of some abelian surfaces over finite fields.
\end{Rk}

This paper is organized as follows: In Section~\ref{prelim}, we recall some of the facts in the theory of Jordan groups together with the classification results for the automorphism groups of abelian surfaces over finite fields (for reader's convenience). In Section~\ref{main}, we obtain the desired results (Theorems \ref{main theorem1} and \ref{main theorem2} above) using the facts that are introduced in Section \ref{prelim}. \\

\indent In the sequel, let $g \geq 1$ be an integer and let $q=p^a$ for some prime number $p$ and an integer $a \geq 1,$ unless otherwise stated. Also, for an integer $n \geq 1,$ we denote by $C_n$ (resp.\ $D_n$, resp.\ $S_n$, resp.\ $\textrm{Dic}_{4n}$) the cyclic group of order $n$ (resp.\ the dihedral group of order $2n$, resp.\ the symmetric group on $n$ letters, resp.\ the dicyclic group of order $4n$). Finally, $\mathfrak{T}^*$ (resp.\ $\mathfrak{O}^*$, resp. $\mathfrak{I}^*$) denotes the binary tetrahedral group (resp.\ binary octahedral group, resp.\ binary icosahedral group).

\section{Preliminaries}\label{prelim}

\subsection{Jordan constants}
In this subsection, we introduce briefly some basic facts from the theory of the Jordan constants of groups. \\

We begin with the following easy but very useful fact.
\begin{Lem}\cite[Lemma 2.2]{Hwa2022}\label{sup lem}
Let $G$ be a Jordan group. Then every subgroup $H$ of $G$ is a Jordan group and we have
\begin{equation*}
J_G = \sup_{H \leq G} J_H
\end{equation*}
where the supremum is taken over all finite subgroups $H$ of $G.$
\end{Lem}

Later in the paper, we need to know the values of the Jordan constants of certain direct products of various groups. The next lemma deals with one easy case of that aspect.
\begin{Lem} \label{direct prod lem ab}
Let $A$ be an abelian group and let $H$ be a Jordan group. Then we have $J_{A \times H} = J_H.$
\end{Lem}
\begin{proof}
This easily follows from \cite[Theorem 3]{Pop2014} and the fact that $J_A=1.$
\end{proof}

\begin{Rk}
If $G$ is a finite group, then $G$ is abelian if and only if $J_G =1.$ The same statement does not seem to hold if $G$ is an infinite group (for example, see \cite[$\S4$]{Mos2011}).
\end{Rk}
The following lemma on the normal abelian subgroups of a direct product of two (more general) groups will be used several times later in Section \ref{main}.
\begin{Lem}\label{direct general lem} 
Let $G_1$ and $G_2$ be finite groups, and let $G=G_1 \times G_2$. If $H$ is an abelian normal subgroup of $G,$ then $H$ is a subgroup of the group $\pi_1 (H) \times \pi_2 (H),$ where $\pi_i \colon G \rightarrow G_i$ is the natural projection map for each $i=1,2.$ In other words, $H$ is a subgroup of a direct product of normal abelian subgroups of $G_1$ and $G_2.$
\end{Lem}
\begin{proof}
Note that $\pi_i (H)$ is a normal subgroup of $G_i =\pi_i (G)$ for each $i=1, 2$ since $\pi_i$ are surjective, and $\pi_i (H)$ are abelian since $H$ is abelian. Since $H \leq G_1 \times G_2$, every element $h \in H$ can be written as $h=(g_1, g_2)$ for some $g_i \in G_i,$ and hence, we have $h=(\pi_1 (h), \pi_2 (h)) \in \pi_1 (H) \times \pi_2 (H).$
\end{proof}

Our last preliminary fact in this subsection is on the value of Jordan constant of some simple semidirect product of two cyclic groups.
\begin{Lem}\label{semiprod lem}
For any integer $m \geq 3$, we have $J_{C_m \rtimes C_2} =2.$
\end{Lem}
\begin{proof}
This follows from the observation that $C_m$ is an abelian normal subgroup of $G:=C_m \rtimes C_2$ with $[G:C_m]=2.$
\end{proof}

\subsection{Automorphism groups of polarized abelian surfaces}\label{aut gp}

For convenience, we summarize some known results on the classification of maximal finite groups that can be realized as the automorphism groups of some (non-simple) polarized abelian surfaces over finite fields. Throughout this subsection, let $G$ be a finite group. \\

First, for the case of a product of two non-isogenous elliptic curves, we note:
\begin{Thm}\cite[Theorem 6.6]{Hwa2021}\label{noniso case}
There exist a finite field $k$ and two non-isogenous elliptic curves $E_1$ and $E_2$ over $k$ such that $G$ is the automorphism group of a polarized abelian surface over $k$, which is maximal in the isogeny class of $X := E_1 \times E_2$ if and only if $G$ is one of the $14$ groups in the following list (up to isomorphism): \\
\normalfont{(1)-(3)} $C_2 \times C_r$ for $r \in \{2, 4, 6 \}$; \\
(4) $C_2 \times \textrm{Dic}_{12}$; \\
(5) $C_2 \times \mathfrak{T}^*$; \\
(6)-(7) $C_4 \times C_r$ for $r \in \{4, 6 \}$; \\
(8) $C_4 \times \textrm{Dic}_{12}$; \\
(9) $C_4 \times \mathfrak{T}^*$; \\
(10) $C_6 \times C_6$; \\
(11) $C_6 \times \textrm{Dic}_{12}$; \\
(12) $C_6 \times \mathfrak{T}^*$; \\
(13) $\textrm{Dic}_{12} \times \textrm{Dic}_{12}$; \\
(14) $\mathfrak{T}^* \times \mathfrak{T}^*.$
\end{Thm}

If we consider the case of an abelian surface that is a power of an ordinary elliptic curve, then we obtain:
\begin{Thm}\cite[Theorem 6.8]{Hwa2021}\label{ordinary case}
There exists a finite field $k$ and an ordinary elliptic curve $E$ over $k$ such that $G$ is the automorphism group of a polarized abelian surface over $k$, which is maximal in the isogeny class of $X := E^2 $ if and only if $G$ is one of the $9$ groups in the following list (up to isomorphism): \\
\normalfont{(1)-(2)} $D_r$ for $r \in \{4, 6 \}$; \\
(3) $\textrm{Dic}_{12}$; \\
(4) $SL_2(\mathbb{F}_3) \cong \mathfrak{T}^*$; \\
(5) $C_{12} \rtimes C_2 \cong C_4 \times S_3$; \\
(6) $GL_2(\mathbb{F}_3)$; \\
(7) $(C_6 \times C_6) \rtimes C_2 \cong \textrm{Dic}_{12} \rtimes C_6$; \\
(8) $\mathfrak{T}^* \times C_3$; \\
(9) $\mathfrak{T}^* \rtimes C_4$.
\end{Thm}

Finally, for the most complicated case of an abelian surface being a power of a supersingular elliptic curve, we have:
\begin{Thm}\cite[Theorem 6.9]{Hwa2021}\label{supersing case}
There exists a finite field $k$ and a supersingular elliptic curve $E$ over $k$ (all of whose endomorphisms are defined over $k$) such that $G$ is the automorphism group of a polarized abelian surface over $k$, which is maximal in the isogeny class of $X := E^2 $ if and only if $G$ is one of the $20$ groups in the following list (up to isomorphism): \\
\normalfont{(1)} $2_{-}^{1+4}.\textrm{Alt}_5 $;  \\
(2) $SL_2(\F_3) \times S_3 $;  \\
(3) $(SL_2(\F_3))^2 \rtimes S_2 $;  \\
(4) $SL_2(\F_9)  $;  \\
(5) $C_3 : (SL_2(\F_3).2) $;  \\
(6) $(\textrm{Dic}_{12})^2 \rtimes S_2 $;  \\
(7) $SL_2(\F_5).2  $;  \\
(8) $ SL_2(\F_5):2 $;  \\
(9) $(Q_8 \rtimes C_3) \rtimes C_2 \cong GL_2 (\F_3)$;  \\
(10) $C_{12} \rtimes C_2 \cong C_4 \times S_3 $;  \\
(11) $(Q_8 \rtimes C_3) \cdot C_4 \cong \mathfrak{T}^* \rtimes C_4$;  \\
(12) $ (C_6 \times C_6) \rtimes C_2 \cong \textrm{Dic}_{12} \rtimes C_6$;  \\
(13) $(Q_8 \rtimes C_3) \times C_3 \cong \mathfrak{T}^* \times C_3$;  \\
(14)-(15) $ D_r$ for $r \in \{4, 6\}$;  \\
(16) $\mathfrak{I}^* \cong SL_2 (\F_5)$; \\
(17) $\textrm{Dic}_{24}$;  \\
(18) $\mathfrak{O}^*$;  \\
(19) $ \mathfrak{T}^* \cong SL_2 (\F_3)$;    \\
(20) $\textrm{Dic}_{12}.$
\end{Thm}

\section{Main Result}\label{main}
In this section, we provide a list of possible values of Jordan constants of the multiplicative subgroups of endomorphism algebras of arbitrary abelian surfaces over finite fields. \\

We begin with recalling the case of simple abelian surfaces.

\begin{Thm}\cite[Theorems 3.18 and 3.19]{Hwa2022} Let $X$ be a simple abelian surface over a finite field $k=\F_q$ with $q=p^a$ for some prime $p$ and an integer $a \geq 1$, and let $D=\textrm{End}_k^0(X)$. Then $J_{D^{\times}}$ is contained in the set $\{1, 2, 12, 24, 60 \}.$ Conversely, if $n \in \{1, 2, 12, 24, 60 \},$ then there is a simple abelian surface $X$ for some finite field $k=\F_q$ ($q=p^a$) such that $J_{D^{\times}}=n$, where $D=\textrm{End}_k^0 (X).$
\end{Thm}
Now, to deal with the case of an abelian surface that is a product of non-isogenous elliptic curves, one of whose endomorphism algebras is commutative, we record the following easy observation.

\begin{Rk}\label{ell rmk}
Let $E_1$ and $E_2$ be two non-isogenous elliptic curves over a finite field $k.$ Let $D_1 = \textrm{End}_k^0(E_1)$ and $D_2 = \textrm{End}_k^0(E_2).$ Then in view of Lemma \ref{direct prod lem ab}, if either $D_1$ or $D_2$ is commutative, then we have $J_{D_1^{\times} \times D_2^{\times}} = J_{D_1^{\times}} \cdot J_{D_2^{\times}}.$
\end{Rk}

For the exact values of Jordan constants of other specific groups in Theorem \ref{noniso case} to work with, we record the following lemmas.

\begin{Lem}\label{dic lem}
For any integer $n \geq 2$, let $G=\textrm{Dic}_{4n} \times \textrm{Dic}_{4n}.$ Then we have $J_{G} = 4.$
\end{Lem}
\begin{proof}
This follows from Lemma \ref{direct general lem} and the fact that $C_{2n}$ is a maximal normal abelian subgroup of $\textrm{Dic}_{4n}.$
\end{proof}

\begin{Lem}\label{T* lem}
Let $G=\mathfrak{T}^* \times \mathfrak{T}^*.$ Then we have $J_G = 144.$
\end{Lem}
\begin{proof}
This follows from Lemma \ref{direct general lem} and the fact that $C_{2}$ is a maximal abelian normal subgroup of $\mathfrak{T}^*.$ (In particular, we have $J_{\mathfrak{T}^*} = 12$.)
\end{proof}




Now, we are ready to introduce one of our main results of this section.
\begin{Thm}\label{prodnonisoellip}
Let $E_1$ and $E_2$ be two non-isogenous elliptic curves a finite field $k=\mathbb{F}_q$ where $q=p^a$ for a prime $p$ with $D_1 = \textrm{End}_k^0(E_1)$ and $D_2 = \textrm{End}_k^0 (E_2)$, and let $X=E_1 \times E_2$ with $D:=\textrm{End}_k^0(X)$. Then $G:=D^{\times}$ is a Jordan group and we have
\begin{equation*}
	J_{G} \in \begin{cases} \{1, 12, 144 \} & \mbox{if $p=2$}; \\ \{1, 2, 4\}  & \mbox{if $p=3$}; \\ \{1\}  & \mbox{if $p \geq 5$} .\end{cases}
	\end{equation*}
\end{Thm}
\begin{proof}
Note first that $D=D_1 \times D_2$ and $G=D_1^{\times} \times D_2^{\times}.$ Also, we have the following three cases for $D_1$ and $D_2$ (up to permutation of factors) to consider: \\
(i) $D_1= \Q(\sqrt{-d_1})$ and $D_2 = \Q(\sqrt{-d_2})$ for some square-free positive integers $d_1$ and $d_2.$ \\
(ii) $D_1 = \Q(\sqrt{-d})$ and $D_2 = D_{p,\infty}$ for some square-free positive integer $d$. \\
(iii) $D_1 = D_2 = D_{p,\infty}.$ \\
In view of \cite[Lemma 3.4]{Hwa2022}, we know that $D_{p,\infty}^{\times}$ is a Jordan group, and hence, it follows from \cite[Theorem 3]{Pop2014} that $G$ is also a Jordan group in all the three cases in the above. \\

Now, we compute the desired Jordan constants explicitly. For the cases (i) and (ii), by Lemma \ref{sup lem}, Remark \ref{ell rmk} and \cite[Lemmas 3.7 and 3.8]{Hwa2022}, we obtain that $J_G \in \{1,2,12\}.$ \\ 
For case (iii), we note that all the finite subgroups of $D_{p,\infty}^{\times}$ are cyclic for $p \geq 5.$ Hence it suffices to consider the cases when $p=2$ or $p=3.$ If $p=3$ (resp.\ $p=2$), then we have $J_G = 4$ (resp. $J_G =144$) by Lemma \ref{dic lem} (resp.\ Lemma \ref{T* lem}). This completes the proof.
\end{proof}

Our next task is to obtain the Jordan constants of some finite subgroups of $GL_2(\C)$.
\begin{Lem}\label{dih lem}
For any integer $n \geq 4,$ we have $J_{D_n} = 2.$
\end{Lem}
\begin{proof}
This immediately follows from the fact that $C_n \leq D_n.$
\end{proof}
Similar argument is applied to prove for the groups $G=C_{12} \rtimes C_2$ and $G=(C_6 \times C_6) \rtimes C_2$ in Theorem \ref{ordinary case}. For the other groups in the same list, we also obtain the following computation.
\begin{Lem}\label{other lem1}
Let $G =GL_2 (\F_3)$ or $G=\mathfrak{T}^* \rtimes C_4$ (resp.\ $G=\mathfrak{T}^* \times C_3$). Then we have $J_G = 24$ (resp.\ $J_G =12$).
\end{Lem}
\begin{proof}
We consider each group one by one. \\
(i) For $G=GL_2(\F_3)$, it is well known that the only proper normal subgroups of $G$ are $C_1, C_2, Q_8, SL_2(\F_3)$, among which only $C_1$ and $C_2$ are abelian. Hence we get $J_G = 24.$  \\
(ii) For $G=\mathfrak{T}^* \times C_3,$ we see that $C_2 \times C_3 \cong C_6$ is the maximal abelian normal subgroup of $G$ by Lemma \ref{direct general lem}, and hence, we get $J_G = 12.$ \\
(iii) For $G=\mathfrak{T}^* \rtimes C_4,$ we know from \cite{GN(1)} that the only abelian normal subgroups of $G$ are $C_1, C_2, C_4,$ and hence, we get $J_G = 24.$ This completes the proof.
\end{proof}

Combining all these computations, we obtain:
\begin{Thm}\label{powordelli}
Let $E$ be an ordinary elliptic curve over a finite field $k=\mathbb{F}_q$  where $q=p^a$ for a prime $p$, and let $X=E^2$ with $D:=\textrm{End}_k^0(X).$ Then $G:=D^{\times}$ is a Jordan group and we have
$$ J_G  \in \{ 2, 12, 24\}.$$ 
\end{Thm}

\begin{proof}
Note first that $D=M_2(\Q(\sqrt{-d}))$ for some square-free positive integer $d.$ Thus we see that $G$ is Jordan in view of the fact that $GL_2(\C)$ is Jordan and $G \leq GL_2(\C)$. \\

The desired result on the Jordan constants follows from \cite[Lemmas 3.7 and 3.8]{Hwa2022}, Theorem \ref{ordinary case}, Lemmas \ref{sup lem}, \ref{dih lem}, and \ref{other lem1}. \\
\indent This completes the proof.
\end{proof}

For our last result, we begin with the following interesting lemma on the value of the Jordan constant of a relatively large group.
\begin{Lem}\label{extraspecial lem}
Let $G=2_{-}^{1+4}.\textrm{Alt}_5$, and let $H$ be a nontrivial normal abelian subgroup of $G.$ Then $H=C_2$, and in particular, we have $J_G = 960.$
\end{Lem}
\begin{proof}
Note first that there is a short exact sequence
$$
\xymatrix{
1 \ar[r] &2_{-}^{1+4} \ar[r]&  G \ar[r]^{\phi} & \textrm{Alt}_5 \ar[r] & 1}
$$
by the definition of the group $G.$ \\

Let $N$ be a (nontrivial) abelian normal subgroup of $G$. Then the image $\phi(N)$ of $N$ under $\phi$ is an abelian normal subgroup of $\textrm{Alt}_5$, and hence, we get $\phi(N)=\{0\}.$ It follows that $N \leq K:= \ker \phi \cong 2_{-}^{1+4},$ and hence, $Z:=Z(K)=C_2$ and $K/Z$ is an elementary abelian $2$-group i.e.\ a vector space $V$ of dimension $4$ over $\F_2.$ \\
Now, every automorphism of $G$ preserves $K$ as a whole, whence preserves $Z$ as a whole so that it fixes $Z$ elementwise. Thus we have $Z \subseteq Z(G)$ and $Z$ is normal and abelian. It follows that if $N$ is a (nontrivial) maximal abelian normal subgroup of $G$, then $N$ is a subgroup of $K$ containing $Z$ and it can be identified with its image in $V$ via the natural projection $\pi \colon K \rightarrow K/Z=V.$ \\
Now, we have an action of $\textrm{Alt}_5$ on $V$ (by lifting to an element of $G$ and conjugating). Note that this representation on $V$ is not trivial, and hence, it is irreducible. Note also that $\textrm{Alt}_5$ cannot act non-trivially on $\F_2^3$ or $\F_2^2$ because an element of order $5$ in $\textrm{Alt}_5$ cannot do so. Hence we have either $\pi(N)=\{0\}$ or $\pi(N)=V,$ or equivalently, either $N=Z$ or $N=K$. Since $K$ is not abelian, it follows that $N=Z$ i.e.\ the only maximal abelian normal subgroup of $G$ equals $Z$, which is isomorphic to $C_2.$ The second assertion immediately follows. \\

This completes the proof.
\end{proof}

The next three lemmas deal with the finite groups, which seem to appear only in small characteristic.
\begin{Lem}\label{ch 2 lem}
Let $G=SL_2(\F_3) \times S_3$ (resp.\ $G=SL_2(\F_3))^2 \rtimes S_2$). Then we have $J_G = 24$ (resp.\ $J_G = 288$).
\end{Lem}
\begin{proof}
(i) The case of $G=SL_2(\F_3) \times S_3$ follows from Lemma \ref{direct general lem}. \\
(ii) For $G=(SL_2(\F_3))^2 \rtimes S_2,$ let $B=\mathfrak{T}^* \times \mathfrak{T}^*$, and let $N$ be a (nontrivial) abelian normal subgroup of $G$. Then we have $[N : B \cap N] \leq [G:B]=2.$ Now, if $[N : B \cap N ] = 2,$ then $N$ cannot be abelian, and hence, we get $[N: B \cap N] = 1$, whence $N \leq B.$ It follows that a maximal such $N$ equals $C_2 \times C_2 = Z(B)$, which in turn, implies that $J_G = 288.$

This completes the proof.
\end{proof}

\begin{Lem}\label{ch 3 lem}
Let $G=SL_2(\F_9)$ (resp.\ $G=C_3:(SL_2(\F_3).2)$, resp.\ $G=\textrm{Dic}_{12}^2 \rtimes S_2$). Then we have $J_G = 360$ (resp.\ $J_G=24$, resp.\ $J_G=8$).
\end{Lem}
\begin{proof}
(i) The case of $G=SL_2(\F_9)$ follows from a well-known fact that $C_2 =Z(G)$ is the maximal abelian normal subgroup of $G.$  \\
(ii) For $G=C_3:(SL_2(\F_3).2),$ we know from \cite{GN(2)} that the only normal abelian subgroups of $G$ are $C_1, C_2, C_3, C_6,$ and hence, we get $J_G = 24.$ \\
(iii) For $G=\textrm{Dic}_{12}^2 \rtimes S_2,$ we know from \cite{GN(3)} that the only abelian normal subgroups of $G$ are $C_1, C_2, C_2^2, C_3^2, C_3 \times C_6, C_6^2,$ and hence, we get $J_G = 8.$

This completes the proof.
\end{proof}

\begin{Lem}\label{ch 5 lem}
Let $G=SL_2(\F_5).2$ or $G=SL_2(\F_5):2.$ Then we have $J_G = 120.$
\end{Lem}
\begin{proof}
In both cases, the only abelian normal subgroups of $G$ are $C_1, C_2$ (see \cite{GN(4),GN(5)}), and hence, it follows that $J_G = 120$.
\end{proof}

Finally, using all these results, we have:
\begin{Thm}\label{pow of supell}
 Let $E$ be a supersingular elliptic curve over a finite field $k=\mathbb{F}_q$ where $q=p^a$ for a prime $p$ (all of whose endomorphisms are defined over $k$), and let $X=E^2$ with $D:=\textrm{End}_k^0(X).$ Then $G:=D^{\times}$ is a Jordan group and we have
$$ J_G  \in \{12, 24, 60, 120, 360, 960 \}.$$
\end{Thm}

\begin{proof}
Note first that $D=M_2(D_{p,\infty})$. Thus we see that $G$ is Jordan in view of the fact that $G \leq GL_2(D_{p,\infty}) \leq GL_4(\C)$ and $GL_4(\C)$ is Jordan. \\

The desired result on the Jordan constants follows from \cite[Lemmas 3.7 and 3.8]{Hwa2022}, Theorem \ref{supersing case}, Lemmas \ref{sup lem}, \ref{dih lem}, \ref{other lem1}, \ref{extraspecial lem}, \ref{ch 2 lem}, \ref{ch 3 lem}, and \ref{ch 5 lem}.

\indent This completes the proof.
\end{proof}

In the rest of this section, we give several remarks and one question that are related to our results.

\begin{Rk}
It might be interesting and not too difficult to give a more precise description in Theorems \ref{powordelli} and \ref{pow of supell} on the Jordan constants as in Theorem \ref{prodnonisoellip} in terms of the characteristic of the base finite field.
\end{Rk}

\begin{Rk}
Using \cite[Theorems 13.4.2, 13.4.4, and 13.4.5]{BH2004} and similar arguments as in the proof of Theorems \ref{prodnonisoellip} and \ref{powordelli}, it is possible to obtain an analogous (but much simpler) computational result on the exact values of the Jordan constants of the full automorphism groups $\textrm{Aut}(X)$ for $2$-dimensional complex abelian varieties $X$.
\end{Rk}

For higher dimensional cases of our main problem, Problem \ref{problem}, we expect the following for now.
\begin{Rk}
In view of a recent work \cite{Hwa2022-2}, it might be possible to answer Problem \ref{problem} for $g=3$ (with $k$ being a finite field) by performing similar but considerably more complicated computations.
\end{Rk}

We conclude this paper with one seemingly natural question motivated by our results.
\begin{Ques}
Is there any other example of a projective variety $X$ (other than an abelian variety) over a field $k$ of positive characteristic for which we are able to compute the exact value of the ``classical" Jordan constant of its full automorphism group?
\end{Ques}

\section*{Acknowledgements}
 The authors sincerely thank Professor Michael Larsen and Professor Gabriele Nebe for invaluable help on the proof of Lemma~\ref{extraspecial lem}.


\begin{thebibliography}{alpha}


\bibitem[1]{CB2016} C. Birkar, {\em Singularities of linear systems and boundedness of Fano varieties,} {\em Ann. of Math.} {{\bf 193} (2021), no. 2, 347–-405.}

\bibitem[2]{BH2004} C. Birkenhake and H. Lange, {\em Complex Abelian Varieties}, (Springer, New York, 2004, 2nd edition).


\bibitem[3]{Hu2020} F. Hu, {\em Jordan property for algebraic groups and automorphism groups of projective varieties in arbitrary characteristic,} {\em Indiana Univ. Math. J.} {{\bf 69} (2020), no. 7, 2493--2504.}

\bibitem[4]{Hwa2021} W. Hwang, {\em On a classification of the automorphism groups of polarized abelian surfaces over finite fields,} {\em Finite Fields Appl.} {{\bf 72} (2021), no. 101809, 30pp.}

\bibitem[5]{Hwa2022} W. Hwang, {\em Jordan constants of quaternion algebras over number fields and simple abelian surfaces over fields of positive characteristic,} {\em Math. Nachr.} {{\bf 295} (2022), no. 3, 560–-573.}

\bibitem[6]{Hwa2023} W. Hwang, {\em Automorphism Groups of Simple Abelian Fourfolds over Finite Fields and Jordan Constants,} {\em To appear in Transform. Groups.}

\bibitem[7]{Hwa2022-2} W. Hwang, B. Im, and H. Kim, {\em A classification of the automorphism groups of polarized abelian threefolds over finite fields.} {\em Finite Fields Appl.} {{\bf 83} (2022), no. 102082, 37pp.}



\bibitem[8]{GN(1)}
T.~Dokchitser, \textit{Group{N}ames}, {a}vailable at:
  \url{https://people.maths.bris.ac.uk/~matyd/GroupNames/73/U(2,3).html} [last accessed
  {J}uly 2022].

\bibitem[9]{GN(2)}
T.~Dokchitser, \textit{Group{N}ames}, {a}vailable at:
  \url{https://people.maths.bris.ac.uk/~matyd/GroupNames/129/C6.5S4.html} [last accessed
  {J}uly 2022].

\bibitem[10]{GN(3)}
T.~Dokchitser, \textit{Group{N}ames}, {a}vailable at:
  \url{https://people.maths.bris.ac.uk/~matyd/GroupNames/288/Dic3wrC2.html} [last accessed
  {J}uly 2022].

\bibitem[11]{GN(4)}
T.~Dokchitser, \textit{Group{N}ames}, {a}vailable at:
  \url{https://people.maths.bris.ac.uk/~matyd/GroupNames/217/CSU(2,5).html} [last accessed
  {J}uly 2022].


\bibitem[12]{GN(5)}
T.~Dokchitser, \textit{Group{N}ames}, {a}vailable at:
  \url{https://people.maths.bris.ac.uk/~matyd/GroupNames/217/C2.S5.html} [last accessed
  {J}uly 2022].

\bibitem[13]{LP2011} M. J. Larsen and R. Pink, {\em Finite subgroups of algebraic groups,} {\em J. Amer. Math. Soc.} {{\bf 24} (2011), no. 4, 1105–-1158.}

\bibitem[14]{MD2018} S. Meng and D. Zhang, {\em Jordan properties for non-linear algebraic groups and projective varieties,} {\em Amer. J. Math.} {{\bf 140} (2018), no. 4, 1133–-1145.}


\bibitem[15]{Mos2011} L. Moser-Jauslin, {\em Automorphism groups of Koras–Russell threefolds of the first kind,} in {\em Affine
Algebraic Geometry}. CRM Proceedings and Lecture Notes, vol. 54 (American Mathematical Society, Providence, 2011), 261-–270

\bibitem[16]{Mum2008} D. Mumford, {\em Abelian Varieties,} (Hindustan Book Agency, 2008, Corrected reprint of the 2nd edition).

\bibitem[17]{Popov2011} V. L. Popov, {\em On the Makar-Limanov, Derksen invariants, and finite automorphism groups of algebraic varieties,} in {\em Affine
Algebraic Geometry}. CRM Proceedings and Lecture Notes, vol. 54 (American Mathematical Society, Providence, 2011), 289-–311

\bibitem[18]{Pop2014} V.~L. Popov, {\em Jordan groups and automorphism groups of algebraic
  varieties,} (I.~Cheltsov et~al., eds.), \textit{Automorphisms in birational
  and affine geometry}, Springer, Cham, 185--213.

\bibitem[19]{PS2016} Y. Prokhorov and C. Shramov, {\em Jordan property for Cremona groups,} {\em Amer. J. Math.} {{\bf 138} (2016), no. 2, 403–-418.}

\bibitem[20]{PS2017} Y. Prokhorov and C. Shramov, {\em Jordan constant for Cremona group of rank 3,} {\em Mosc. Math. J.} {{\bf 17} (2017), no. 3, 457–-509.}

\bibitem[21]{PS2018} Y. Prokhorov and C. Shramov, {\em Finite groups of birational selfmaps of threefolds,} {\em Math. Res. Lett.} {{\bf 25} (2018), no. 3, 957–-972.}

\bibitem[22]{Ser2009} J. P. Serre, {\em A Minkowski-style bound for the orders of the finite subgroups of the Cremona group of rank 2 over an arbitrary
field,} {\em Mosc. Math. J.} {{\bf 9} (2009), no. 1, 193–-208.}

\bibitem[23]{Ya2017} E. Yasinsky, {\em The Jordan constant for Cremona group of rank 2,} {\em Bull. Korean Math. Soc.} {{\bf 54} (2017), no. 5, 1859–-1871.}





















\end{thebibliography}
\end{document}